\documentclass[11pt]{amsart}
\usepackage{amssymb,amsmath,txfonts}
\usepackage{hyperref}
\newtheorem{theorem}{Theorem}

\newtheorem{remark}{Remark}
\newtheorem{claim}{Claim}
\newtheorem*{definition}{Definition}

\def\Xint#1{\mathchoice
  {\XXint\displaystyle\textstyle{#1}}%
  {\XXint\textstyle\scriptstyle{#1}}%
  {\XXint\scriptstyle\scriptscriptstyle{#1}}%
  {\XXint\scriptscriptstyle\scriptscriptstyle{#1}}%
  \!\int}
\def\XXint#1#2#3{{\setbox0=\hbox{$#1{#2#3}{\int}$}
  \vcenter{\hbox{$#2#3$}}\kern-.5\wd0}}

\def\dashint{\Xint-}

\author{Gang Liu}
\address{Department of Mathematics\\University of Minnesota\\Minneapolis, MN 55455}
\email{liuxx895@math.umn.edu}
\title[Rigidity of volume entropy]{\bf A short proof to the rigidity of volume entropy}
\date{}
\begin{document}

\begin{abstract}In this note we give a short proof to the rigidity of volume entropy.
The result says that for a closed manifold with Ricci curvature bounded from below,
if the universal cover has maximal volume entropy, then it is the space form.
This theorem was first proved by F. Ledrappier and X. Wang in [1].
\end{abstract}
\maketitle

\begin{definition}{ For a complete Riemannian manifold $M$, define the volume entropy $v$ of $M$ as
$$v(M) = \overline{\lim\limits_{r \to \infty}}\frac{\ln vol B_{M}(x, r)}{r}$$ where $B_{M}(x, r)$ is the geodesic
 ball of radius $r$ centered at $x$ in $M$. }
\end{definition}
 We are going to prove the following theorem due to F. Ledrappier and X. Wang in [1]:

\begin{theorem}
{ Let $(M^n, g)$ be a closed Riemannian manifold with $Ric \geq -(n-1)$. Let $\tilde M$ be its universal cover,
  then the volume entropy satisfies $v(\tilde M) \leq n -1$.
Moreover, $v(\tilde M) = n -1$ iff $\tilde{M}$ is the standard hyperbolic space with constant curvature $-1$.}
\end{theorem}
\begin{proof}
The inequality $v(\tilde M) \leq n -1$ directly follows from the volume comparison. We have to deal with the equality case.
We shall construct a Busemann function $u$ on $\tilde M$ such that $\Delta u = n -1$ in the distribution sense.
By the result of Li-Wang in [2], we know $\tilde M$ is the hyperbolic space form since $\tilde M$ has bounded curvature.
Now take a fixed $R$ such that $R > 50 diam(M)$. Pick a point $O \in \tilde{M}$ and define $r(x) = d(O, x)$.

\begin{claim}\emph{ There exists a sequence $r_i \to \infty$ so that the area of the geodesic spheres satisfy
 $$\frac{A(\partial(B(O, r_i + 50R)))}{A(\partial (B(O, r_i - 50R)))} \to e^{100(n-1)R}.$$}
\end{claim}

We prove the claim by contradiction. Suppose there exist $r_0 > 100R > 0$ and $\epsilon >0$ such that for any $r >r_0$,
$$\frac{A(\partial(B(O, r + 50R)))}{A(\partial (B(O, r - 50R)))} \leq e^{100(n-1)R} (1 - \epsilon).$$
 By an iteration argument we find that for sufficiently large $r$, $$A(\partial(B(O, r))) \leq C(1-\epsilon)^{\frac{r}{100R}}e^{(n-1)r}$$ where $C$ is a constant independent of $r$.
After the integration, we find that the volume entropy is smaller than $n-1$. This is a contradiction.

\bigskip

We take the sequence $r_i$ in claim 1 and define $$A_i = \{x\in \tilde M|r_i - 50R \leq d(x, O) \leq r_i+50R\}.$$

\begin{claim}\emph{ $\dashint_{A_i} \Delta r \geq n-1-\epsilon(i, R)$ where $\epsilon(i, R) \to 0$ when $i \to \infty$. The symbol $\dashint$ means the average.}
\end{claim}
After integration by parts,  claim 2 follows from claim 1 and Bishop-Gromov's volume comparison.

\bigskip

Given a point $P \in M$, for all preimages of $P$ in $\tilde M$, consider the subset
$P_{j}(i)$ such that $B(P_{j}(i), R) \subseteq A_i$. Let $E_i$ be the maximal set of $P_{j}(i)$ such that
for $j_1 \neq j_2$ in $E_i$,
 $B(P_{j_1}(i), R) \bigcap B(P_{j_2}(i), R) = \Phi$.
Take $F_i = \bigcup\limits_{j \in E_i} B(P_{j}(i), R), G_i = \bigcup\limits_{j \in E_i} B(P_{j}(i), 5R)$.

By Bishop-Gromov's volume comparison, we have $$\frac{vol(F_i)}{vol(G_i)} \geq g(R, n).$$

Now by a standard covering technique, we find that $$G_i = \bigcup\limits_{j \in E_i} B(P_{j}(i), 5R) \supseteq B(O, r_i + 10R)\backslash B(O, r_i-10R).$$

Combining with claim 1, we have $$\frac{vol(G_i)}{vol(A_i)} \geq h(R, n)$$
where $g(R, n), h(R, n)$ are positive functions independent of $i$. Therefore, we have
$$\frac{vol(F_i)}{vol(A_i)} \geq g(R, n)h(R, n).$$

Combining with claim 2 and the Laplacian comparison, we find that for each $i$,
$$\dashint_{F_i}\Delta r \geq n-1 - \frac{\epsilon(R, i)}{g(R, n)h(R, n)} - \delta (i, n)$$
where $\delta (i, n) \to 0$ when $i \to \infty$.

Therefore there exists at least one $j$ in $E_i$ such that
\begin{equation}
\dashint_{B(P_{j}(i), R)}\Delta r \geq n - 1 -\frac{\epsilon(R, i)}{g(R, n)h(R, n)} - \delta (i, n).
\end{equation}

Note that $B(P_{j}(i), R)$ is isometric to $B(P_0, R)$ where $P_0$ is a fixed preimage of $P$ in $\tilde M$. Consider the function $u_i(x) = r(x) - d(O, P_j(i))$ in $B(P_{j}(i), R)$, we pull $u_i$ back to $B(P_0, R)$. Note that $u_i(P_0) = 0$.
Since $u_i$ is a uniformly Lipschitz sequence, we can extract a subsequence so that $u_i \to u_R$ in $B(P_0, R)$.
Now by (1) and the Laplacian comparison, we can easily get
\begin{equation}
\int_{B(P_{0}(i), R)}u_R\Delta \varphi  \geq (n - 1) \int_{B(P_{0}(i), R)}\varphi
\end{equation}
for any $\varphi \in C_0^{\infty}(B(P_{0}, R)), \varphi \geq 0$.

One the other hand, since $u_R$ is a limit of the distance function, the standard Laplacian comparison implies
\begin{equation}
\int_{B(P_{0}(i), R)}u_R\Delta \varphi  \leq (n - 1) \int_{B(P_{0}(i), R)}\varphi
\end{equation}
for any $\varphi \in C_0^{\infty}(B(P_{0}(i), R)), \varphi \geq 0$.

(2) and (3) imply $\Delta u_R = n - 1$ in the distribution sense. Furthermore, since $u_R$ is a limit of the distance function, $|\nabla u_R| = 1$.
Let $R \to \infty$, we can extract a subsequence of $u_R$ so that $u_R \to u$.
Then $u$ is defined on $\tilde M$. It satisfies $|\nabla u|= 1$ and $\Delta u= n - 1$.
According to the argument at the beginning of the proof, $\tilde M$ is the hyperbolic space form.
\end{proof}

Using the same proof, we can prove the following theorem which is also due to F. Ledrappier and X. Wang [1]:
\begin{theorem}
{ Let $M$ be a compact K\"{a}hler manifold with $dim_{\mathbb{C}}M = m$ and $\tilde M$ be its universal cover. If the bisectional curvature $K_{\mathbb{C}} \geq -2$, then the volume entropy satisfies $v \leq 2m$. Moreover, if the equality holds iff $\tilde M$ is the complex hyperbolic space form.}
\end{theorem}
\begin{remark}
It is not clear to the author whether theorem 2 still holds if we relax the condition to $Ric \geq -2(m+1)$.
\end{remark}

\section*{Acknowledgements}

The author would like to express his deep gratitude to his advisor, Professor Jiaping Wang, for his interest in this problem and useful suggestions.


\begin{thebibliography}{99}
\bibitem{[1]} F. Ledrappier and X. Wang,
\emph{A integral formula for the volume entropy and an application to rigidity}, to appear in J. Diff. Geom.

\bibitem{[2]}
P. Li and J. Wang, \emph{Complete manifolds with positive spectrum II}, J. Diff. Geom. 62(2002),
143-162.

\end{thebibliography}
\end{document}